\documentclass[11pt, nonatbib]{elsarticle}

\makeatletter
\let\c@author\relax
\makeatother

\usepackage[backend=biber, citestyle=numeric-comp, bibstyle=trad-abbrv, url=false, doi=false, isbn=false]{biblatex}
\usepackage{fancyhdr}
\usepackage{marginnote}

\usepackage{amsmath}
\usepackage{mathtools}
\usepackage{amsfonts}
\usepackage{amsthm}
\usepackage{amssymb}
\usepackage{color}
\usepackage{dsfont}
\usepackage{geometry}
\usepackage{subfigure}

\usepackage{multicol}
\newgeometry{tmargin=2.8cm, bmargin=3.5cm, lmargin=1.7cm, rmargin=1.7cm}

\usepackage{color}

\definecolor{darkgreen}{rgb}{0.00, 0.50, 0.00}

 \def\thebibliography#1{\section*{References\markboth
  {References}{References}}\list
  {[\arabic{enumi}]}{\settowidth\labelwidth{[#1]}\leftmargin\labelwidth
  \advance\leftmargin\labelsep
  \usecounter{enumi}}
  \def\newblock{\hskip .11em plus .33em minus -.07em}
  \sloppy
  \sfcode`\.=1000\relax}

 \let\OLDthebibliography\thebibliography
 \renewcommand\thebibliography[1]{
   \OLDthebibliography{#1}
   \setlength{\parskip}{3pt}
   \setlength{\itemsep}{0pt plus 1\baselineskip}
 }

\bibsetup{
  \setcounter{abbrvpenalty}{0}
  \setcounter{highnamepenalty}{0}
  \setcounter{lownamepenalty}{0}
}

\newtheorem{theo}{\bf Theorem} 
 
\newtheorem{coro}{\bf Corollary}[section]
\newtheorem{lem}[coro]{\bf Lemma} 
\newtheorem{rem}[coro]{\bf Remark} 
\newtheorem{defi}[coro]{\bf Definition}

\newtheorem{prop}[coro]{\bf Proposition}

\def\R{{\mathbb{R}}}

\def\N{{\mathbb{N}}}
\def\rn{{\mathbb{R}^{n}}}

\def\vk{{\varkappa}}

\newcommand{\supp}{{\mathrm{supp}\,}}

\def\rp{{[0,\infty )}}

\def\ve{{\varepsilon}}

\def\vt{{\vartheta}}

\newcommand{\wt}{\widetilde}

\newcommand{\vp}{\varphi}

\newcommand{\cF}{{\mathcal{F}}}
\newcommand{\cE}{{\mathcal{E}}}
\newcommand{\cG}{{\mathcal{G}}}
\newcommand{\cH}{{\mathcal{H}}}

\newcommand{{\SP}}{{\mathcal{Z}}}
\def\Plog{\mathcal{P}^{\text{log}}}

\begin{document}

\begin{frontmatter}

\title{Absence and presence of Lavrentiev's phenomenon\\ for double phase functionals upon every choice of exponents}

\author[1]{Michał Borowski}\ead{m.borowski@mimuw.edu.pl}
\author[1]{Iwona Chlebicka\corref{mycorrespondingauthor}}
\ead{i.chlebicka@mimuw.edu.pl} 
\author[2]{Filomena De Filippis}
\cortext[mycorrespondingauthor]{Corresponding author}
\ead{filomena.defilippis@studenti.univaq.it} 
\author[1]{Błażej Miasojedow}\ead{b.miasojedow@mimuw.edu.pl}

\fntext[myfootnote]{MSC2020: 49J45 (46E30,46E40,46A80)}
\fntext[myfootnote]{M.B. and I.C. are supported by NCN grant 2019/34/E/ST1/00120, B.M. is supported by NCN grant 2018/31/B/ST1/00253, F.D.F. is supported by UnivAQ}

\address[1]{Institute of Applied Mathematics and Mechanics,
University of Warsaw, ul. Banacha 2, 02-097 Warsaw, Poland
}
\address[2]{DISIM, University of L'Aquila, Via Vetoio snc, Coppito, 67100 L'Aquila, Italy
}

\begin{abstract} 

We study classes of weights ensuring the absence and presence of the Lavrentiev's phenomenon for double phase functionals upon every choice of exponents. We introduce a new sharp scale for weights for which there is no Lavrentiev's phenomenon up to a counterexample we provide. This scale embraces the sharp range for $\alpha$-H\"older continuous weights. Moreover, it allows excluding the gap for every choice of exponents $q,p>1$.
\end{abstract}


\end{frontmatter}


\section{Introduction}

We consider the following double-phase functional
\begin{equation}\label{eq:defF}
    \cF[u] = \int_{\Omega}  |\nabla u(x)|^p + a(x)|\nabla u(x)|^{q}\,dx\,,
\end{equation}
over open and bounded $\Omega \subset \rn$, $n> 1$, where $1 \leq p, q < \infty$ and weight $a : \Omega \to [0, \infty)$ is bounded. The functional is designed to model the transition between the region where the gradient is integrable with $p$-th power and the region where it has the higher integrability with $q$-th power. Therefore, we are interested only in~the situation when $p < q$ and $a$ vanishes on some subset of $\Omega$, but $a\not\equiv 0$. Functional $\cF$ and various kinds of its minimizers have been studied since~\cite{zh,Marc-ARMA-89} continued in a vast range of contributions including~\cite{bacomi-cv, colmar, comi, eslev, eslemi,demi,hahato,BS2022} with sharpness discussed in~\cite{badi, eslemi,fomami,zh-86,zh}. More recent developments in this matter may be found in~\cite{BO, filomena, Lukas,BaaBy,basu,Bousquet-Pisa,bmt}. Our main focus is on the approximation properties of the double phase version of Sobolev space, in the spirit of~\cite{yags, Bor-Chl,C-b}, and consequences for the double phase functionals, cf.~\cite{bgs-arma,ibm}. By developments of~\cite{bgs-arma, colmar, bacomi-cv}, revisited in~\cite{ibm}, it is known that the condition
\begin{equation}
    \label{classical-framework}\text{$a\in C^{0,\alpha}(\Omega),\ \alpha \in (0,1],\ $ and $\ p < q \leq p + \alpha$}
\end{equation} is enough for good approximation properties of the energy space by regular functions. This condition is meaningful only provided that $p<q \leq p + 1$, but -- as explained in the paper --  $q$ and $p$ can be actually arbitrarily far from each other as long as the  weight has relevant properties. We also show the sharpness of our new scale up to a counterexample we provide. \newline

Let us at first settle what we mean by the Lavrentiev's phenomenon in our case. For $1<p< q<\infty$ and   $a : \Omega \to [0, \infty)$, we set
\[
M(x,t)=t^p+a(x)t^q.\]
Given a bounded and open set $\Omega\subset\rn$, let us define the energy space\begin{align}\label{eq:Wdef}
    W(\Omega)&:=\left\{\vp\in W^{1,1}_0(\Omega):\quad \int_\Omega M(x,|\nabla \vp|)\,dx<\infty\right\}\,
\end{align}
endowed with a Luxemburg-type norm. 
Note that we have the inclusion $C_c^\infty\subset W$, and in turn
\[\inf_{v\in u_0+W}\cF[v]\leq \inf_{w\in u_0+C_c^\infty}\cF[w]\,.\]
It is known that if $M$ is not regular enough, the inequality above is strict, 
i.e.,
\begin{equation}\label{eq:LavOc}
\inf_{v\in u_0+W}\cF[v]< \inf_{w\in u_0+C_c^\infty}\cF[w]\,,
\end{equation}
which means that the Lavrentiev's phenomenon between spaces $C_c^\infty$ and $W$ occurs. The first example of~such situation, for a different functional, was provided by Lavrentiev, see~\cite{Lav, Mania}. There was a deep interest in the autonomous and nonautonomous problems throughout the decades. See~\cite{Buttazzo-Belloni,Buttazzo-Mizel,eslemi,zh} and references therein as well as already mentioned recent contributions \cite{bacomi-cv, colmar, comi, eslev, eslemi,badi, eslemi,BO, filomena, Lukas,BaaBy,basu,Bousquet-Pisa,bmt,bgs-arma,ibm}. In particular, in~\cite{zh} Zhikov introduced the double-phase functional \eqref{eq:defF} and provide an example of the Lavrentiev's phenomenon in dimension $n=2$ and for $p \in [1,2]$, $q > 3$ and a Lipschitz continuous weight $a$. His example was extended to arbitrary $n \geq 2$ in~\cite{eslemi}, requiring that $p < n < n + \alpha < q$, where $\alpha$ is an exponent of the H\"older continuity of the weight $a$.

The regularity of the possibly vanishing weight $a$ dictates how far apart can be powers $p$ and $q$ to exclude~\eqref{eq:LavOc}. In particular, it was known that if~\eqref{classical-framework} is satisfied, then  there is no Lavrentiev's phenomenon  and $q=p+1$ used to be treated as a borderline. We consider a new sharp scale $\SP^\vk$ that captures the abovementioned result. A function $a\in\SP^\vk$ is assumed to decay in the transition region at least like a power function with an exponent $\vk$ for $\vk\in (0,\infty)$. In turn, our approach extends the result on the range for the absence of the Lavrentiev's gap to $a\in\SP^\vk$ within
\begin{equation}\label{range-no-gamma}
   p< q\leq p+\vk, \quad \vk\in (0,\infty)\,.
\end{equation}  Within the range~\eqref{range-no-gamma}, we prove the absence of the Lavrentiev's phenomenon for $\mathcal{F}$ between $W$ and $C_c^\infty$ up to a~counterexample from Section~\ref{sec:example}. The definition of $\SP^{\vk}$ reads as follows. 
\begin{defi}[Class $\SP^\vk(\Omega)$, $\vk\in (0,\infty)$] Let $\Omega\subset \rn$, $n\in\N$.
 A function $a:\Omega \to [0, \infty)$ belongs to $\SP^\vk(\Omega)$ for $\vk \in (0,\infty)$, if there exists a positive constant $C$ such that
\begin{equation} \label{new-condition}
  a(x) \leq C\left(a(y) + |x-y|^{\vk}\right) \,
\end{equation}
for all $x,y \in \Omega$.
\end{defi} 

\begin{figure}[htbp]\label{fig:fignn}
  \centering
  \includegraphics[width=8cm,height=6cm]{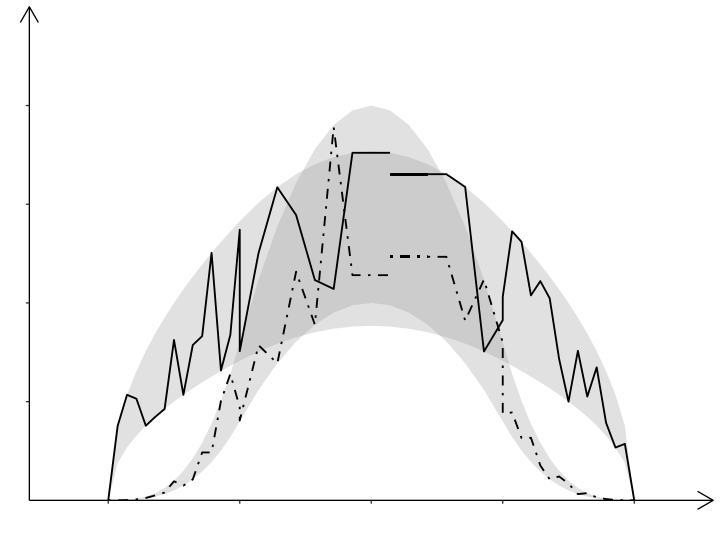}
  \caption{Solid line represents an example of $a_1\in\SP^{\vk_1}$ for $\vk_1\in (0,1)$, while dot-dashed line $a_2\in\SP^{\vk_2}$ for $\vk_2>1$. We stress that $a_2\in\SP^{\vk_2}\subset\SP^{\vk_1}$.}
\end{figure}

Of course, $\alpha$-H\"older continuous functions for $\alpha\in (0,1]$ belong to $\SP^\alpha$, but $\SP^\vk$ with $\vk\in(0,\infty)$ is an~essentially broader class of functions, see Figure~\ref{fig:fignn}. In particular, for every $\vk \in (0, \infty)$, we have that the function $x \mapsto |x|^\vk$ belongs to $\SP^\vk(\rn)$. To provide better understanding of this new scale, we set down its main properties.
\begin{rem}[Basic properties of $\SP^\vk(\Omega)$]\label{rem:zk} If $\Omega \subset \rn$ is an open set, then the following holds.
\begin{enumerate}
    \item A function $a$ belongs to $\SP^\vk(\Omega)$ for $\vk\in(0,1]$ if and only if there exists $\wt a\in C^{0,\vk}(\Omega)$, such that $a$ is comparable to $\wt a$; i.e., there exists a positive constant $c$ such that $\wt a \leq a \leq c\wt a$.
    \item Let $\vk,\beta\in(0,\infty)$. Function $a\in\SP^\vk(\Omega)$ if and only if $a^\beta\in\SP^{\beta\vk}(\Omega)$. 
    \item Function $a\in\SP^\vk(\Omega)$ for $\vk\in(0,\infty)$ if and only if there exists $\wt{a}$ comparable to $a$ such that $\wt{a}^{\frac{1}{\vk}} \in Lip(\Omega)$.
    \item If $0<\vk_1\leq\vk_2$, then $\SP^{\vk_2}(\Omega)\subset\SP^{\vk_1}(\Omega)$. 
\end{enumerate}
\end{rem}
The first point of Remark~\ref{rem:zk} says that for $\vk \in (0, 1]$, class $\SP^\vk(\Omega)$ is similar to H\"older continuity, but it is actually requiring admissible decay rate near regions where $a$ vanishes. The second point of the remark allows extending this intuition to $\vk > 1$, as we can look at some power of $a$. In particular, according to the third point, the $\vk$-th roots of functions in $\SP^{\vk}$ are comparable to Lipschitz continuous functions.  We show examples of functions in $\SP^\vk$ on an interval for large and small values of parameter $\vk$ on Figure~{\ref{fig:fignn}}. In both of~these cases there is no reason for the smoothness or the continuity of functions from $\SP^\vk$. The controlled property is the rate of decay in the transition region, which is comparable to a power function with an exponent $\vk$. {Let us stress that $C^{1,\alpha}$-regularity for $\alpha\in(0,1]$ of the weight implies its $\SP^{1+\alpha}$-regularity, but smoothness of the weight does not give more than $\SP^2$. To state it precisely, we give the following proposition proven in the appendix.
\begin{prop} \label{prop:C1ainSPa+1} If $\Omega \subset \rn$ is an open set, then the following holds.
\begin{enumerate}
        \item[(i)] If $0 \leq a \in C^{1, \alpha}(\overline{\Omega})$ for some $\alpha \in (0, 1]$ and $a>0$ on $\partial \Omega$, then $a \in \SP^{1 + \alpha}(\Omega)$.
        \item[(ii)] 
There exists $0 \leq a \in C^{\infty}(\overline{\Omega}),$ such that $a>0$ on $\partial \Omega$, $a\in \SP^2(\Omega)$, and $a\not\in \SP^{2+\ve}(\Omega)$ for any $\ve>0$.\end{enumerate}
\end{prop}}

Our main result yields that for any $q$ and $p$ the Lavrentiev's phenomenon is absent provided that close to the phase transition $a$ is decaying not slower than a power function with an exponent $\frac{1}{q-p}$. We also claim that this range is sharp, that is if the decay rate is slower and $p$ and $q$ meet the dimensional threshold, then the gap occurs. 
\begin{theo}
    \label{theo:main} Suppose  $1<p < q<\infty$, $\vk>0$, $a:\Omega\to\rp$, $\cF$ is given by \eqref{eq:defF}, and $W(\Omega)$ is defined in~\eqref{eq:Wdef}. Then the following claims hold true. \begin{enumerate}
        \item[(i)] If $q\leq p+\vk$, $\Omega$ is a Lipschitz domain, $a\in \SP^\vk(\Omega)$, and $u_0 $ satisfies $\cF[u_0]<\infty$, then \[ 
\inf_{v\in u_0+W(\Omega)}\cF[v]= \inf_{w\in u_0+C_c^\infty(\Omega)}\cF[w]\,,\] i.e., there is no Lavrentiev's phenomenon between $W(\Omega)$ and $C_c^\infty(\Omega)$.
        \item[(ii)] If $p<n<n+\vk<q$, then there exist a Lipschitz domain $\Omega$, $a\in \SP^\vk(\Omega)$ and $u_0 $ satisfying $\cF[u_0]<\infty$, such that  \[
\inf_{v\in u_0+W(\Omega)}\cF[v]<\inf_{w\in u_0+C_c^\infty(\Omega)}\cF[w]\,,\] i.e., the Lavrentiev's phenomenon occur between $W(\Omega)$ and $C_c^\infty(\Omega)$.
    \end{enumerate}
    \vspace{0.5em}
\end{theo}

In turn, in order to ensure the absence of the Lavrentiev's gap for any $q$ and $p$ one can take a weight $a$ decaying like $e^{-1/t^2}$ that is faster than any polynomial, see Remark~\ref{rem:zk}, {p.~3}. We stress that no continuity or smoothness of $a$ is required. {Nonetheless, in the view of Proposition~\ref{prop:C1ainSPa+1}, if $a\in C^{1,\alpha},$ $\alpha\in(0,1],$ then Theorem~\ref{theo:main} {\it (i)} implies the absence of the Lavrentiev's phenomenon as long as $q\leq p+1+\alpha$.} Let us note that Theorem~\ref{theo:main} is formulated for a model for the sake of clarity of exposition. The same conclusion as in {\it (i)} for a general class of functional is given by Theorem~\ref{theo:lavrentiev-general}. Let us note that a modification of our method might be used to relax the required regularity of the domain, cf.~\cite{bmt}. Moreover, we point out that there are methods to get rid of the dimensional threshold between $p$ and $q$ in $(ii)$ that involve construction of fractals, see~\cite{badi} for a general method. However, we restrict ourselves to the exponents satisfying $p<n<n+\vk<q$ in order to make the proofs as concise and straightforward as possible.   \newline

 On the other hand, we prove also that the functional $\cF$ given by~\eqref{eq:defF}  enjoys interpolation properties. Namely, if one assumes additionally that $u\in C^{0,\gamma}$, $\gamma\in (0,1]$, we can relax the bound \eqref{range-no-gamma} even further. In fact, for $\gamma=1$ there is no gap for arbitrary $p$ and $q$. Moreover, to exclude the gap between $W\cap C^{0,\gamma}$ and $C_c^\infty$, it suffices to take
 \begin{equation}\label{eq:rangegamma}q\leq p+\frac{\vk}{1-\gamma},\quad \vk\in(0,\infty)
 \,.\end{equation}
 This we provide in Section~\ref{ssec:no-gap}, namely in Theorem~\ref{theo:lavrentiev-model} and its extension Theorem~\ref{theo:lavrentiev-general}. As far as double phase functionals are concerned, it fully covers the results of~\cite{bacomi-cv,ibm, bgs-arma}, but allows considering arbitrarily far $p$ and $q$, in case of any $\gamma$.  Moreover, this interpolation phenomenon has one more application. Note that in case of functional~\eqref{eq:defF}, for $p > n$ one can apply Morrey Embedding Theorem to obtain that all functions from energy space $W$ are H\"older continuous with a certain exponent. Substitution of this exponent in the range~\eqref{eq:rangegamma} allows for the condition $q \leq p + \tfrac{p\vk}{n}$, for the absence of the Lavrentiev's phenomenon between $C_c^\infty$ and $W$. This range embraces the classical range from~\cite{eslemi}, but it employs the scale $\SP^{\vk}$ instead of H\"older continuity. Therefore, our range is meaningful for $p$ and $q$ arbitrary far away. We can cover both cases $p \leq n$ and $p > n$ simultaneously, obtaining the following theorem. 
\begin{theo}
    Suppose $\Omega\subset\rn$ is a bounded Lipschitz domain, $p, q, \vk$ are such that \[\vk > 0\quad\text{ and }\quad 1 < p < q \leq p + \vk\max\left\{\tfrac pn, 1\right\}\,,\]  $a:\Omega\to\rp$ is such that $a \in \SP^\vk(\Omega)$, $\cF$ is given by \eqref{eq:defF}, and $W(\Omega)$ is defined in~\eqref{eq:Wdef}. Then for any $u_0$ such that $\cF[u_0] < \infty$, it holds that
    \begin{equation}
        \inf_{v\in u_0+W(\Omega)}\cF[v]= \inf_{w\in u_0+C_c^\infty(\Omega)}\cF[w]\,,
    \end{equation}
    i.e., there is no Lavrentiev's phenomenon between $W(\Omega)$ and $C_c^\infty(\Omega)$.
\end{theo}
One can ask whether the preeminent regularity results for double phase functionals may be proven within the class $\SP^\vk$. Let us point out that assumptions {\it (A1)} and {\it (A1-n)}, studied in \cite{hahab} for analysis in generalized Orlicz spaces, specified for the double phase energy for $a\in \SP^\vk(\Omega)$, $\vk>0$ read $q\leq p+p\vk/n$ and  $q\leq p+\vk$, respectively. In turn, in \cite{hahato} it is proven that quasiminimizers of $\cF$ are H\"older continuous provided {\it (A1)} and {\it (A1-n)} hold, which 
 means $q\leq p+\vk\min\{p/n,1\}$ for $\vk>0$. On the other hand, in \cite{hhl} it is proven that $\omega$-minimizers of $\cF$ are H\"older continuous provided {\it (A1)} or $u$ is a priori bounded and {\it (A1-n)} holds. Moreover, for the double phase functional $\cF$ with $a\in\SP^\vk(\Omega)$ condition {\it (VA1)} from~\cite{ho1,ho2} is also equivalent to $q/p\leq 1+\vk/n$, $\vk>0$. Consequently, $C^{0,\beta}$ and $C^{1,\beta}$-regularity of local minimizers from these papers hold for $a\in\SP^\vk(\Omega)$ with $q/p\leq 1+\vk/n$, $\vk>0$. It would be interesting to extend the 
 results of~\cite{bacomi-cv} to cover  $a\in \SP^\vk(\Omega)$ for all $\vk>0$ and  $q\leq p+\vk$ (if a priori $u\in L^\infty$) or $q\leq p+\vk/(1-\gamma)$ (if a priori $u\in C^{0,\gamma}$).  Note that the $C^{0,\beta}$ and $C^{1,\beta}$-regularity of local minimizers is also the topic of \cite{BaaBy}, where  the functionals are of Orlicz multi phase growth. Nonetheless, the assumption made therein that the weights $a_i\in C^{0,\omega_i}$ for controlled moduli of continuity $\omega_i$ which are concave does not allow for studying Orlicz phases with growths arbitrarily far, which is allowed for counterparts of $\SP^\vk$ for all $\vk>0$ or under {\it (A1)/(A1-n)}. See Section~\ref{sec:gen} for possible extension of $\SP^\vk$ for the Orlicz multi phase case. The double phase functionals with the integrand depending on $(x,u,\nabla u)$ (cf.~\eqref{cG-def}) embracing both of the mentioned contributions allowing for phases with growth arbitrarily far apart and involving the weights decaying like the dot-dashed one in Figure~\ref{fig:fignn} are still calling for the local regularity theory under the sharp regime. \newline 

Let us briefly summarize our methods. Unlike~\cite{Lukas,Bousquet-Pisa,filomena}, we do not analyse the behaviour of minimizers, but we study the approximation properties of a relevant function space. In this regard, we first establish the density of smooth functions in the double phase version of the Sobolev space (Theorem~\ref{theo:approx}). 
In related investigations, employing convolution-based approximation is a common technique, see e.g.~\cite{yags, Bor-Chl, ibm, bgs-arma, C-b,BaaBy,eslemi,eslev,hahab} for various variants. We enclose here -- for the sake of completeness -- concise arguments shortening the reasoning of~\cite{bgs-arma} and its anisotropic extension~\cite{ibm} making use of properties of convolution from Lemma~\ref{lem:dec2}. Let us recall how broad is the class $\SP^\vk$, as shown in Figure~\ref{fig:fignn} and in Remark~\ref{rem:zk}. Note that choosing arbitrary $p$ and $q$, one can easily see a power of $\SP^\vk$-scale for all $\vk>0$ in the proof of Theorem~\ref{theo:approx} (precisely in~\eqref{eq:jan7}). This result of ours is not more powerful than~\cite[Theorem~2]{ibm}, but it gives the true feeling that there is no reason for $p$ and $q$ to be close if only one can adjust the decay of the weight to compensate it. The absence of the Lavrentiev's phenomenon, stated in Theorem~\ref{theo:lavrentiev-general}, is a~consequence of the density of smooth functions  via the ideas inspired by~\cite{bgs-arma,ibm} applying the Vitali convergence theorem. The same method can be applied to a broad family of functionals, see Section~\ref{sec:gen} for several examples. The sharpness of the result on the absence of the Lavrentiev's phenomenon  of Theorem~\ref{theo:main} is confirmed by a counterexample we provide in Section~\ref{sec:example}. We indicate a domain, a boundary condition, and weight $a\in\SP^\vk$ for $\vk$ outside the good range for the approximation result (Theorem~\ref{theo:approx}), for which the infima of $\cF$ differ. The method is inspired by the two-dimensional checkerboard constructions of Zhikov~\cite{zh-86,zh} and its extension in~\cite{eslemi}, but requires essentially more delicate arguments.   In detail, we modify a weight $a \in C^{0,\alpha}$ from~\cite{eslemi} to allow $a^{\frac{1}{\vk}}$ being comparable to a Lipschitz function, so that $a \in \SP^\vk$. See definition of weight $a$ in~\eqref{the-weight} and its property of Lemma~\ref{lemma_1}. The small change is surprisingly powerful and justifies the use of $\SP^\vk$-scale for variational problems involving double phase functionals with arbitrarily far powers.  \newline
  
{\bf Organization of the paper}. In Section~\ref{sec:prelim}, we provide information on the notation and basic tools used in the proofs of our results. Section~\ref{sec:approximation-and-no-gap} is devoted to proofs of results concerning density of smooth functions and the absence of the Lavrentiev's phenomenon, while in Section~\ref{sec:example}, sharpness of these results is discussed. Finally, in Section~\ref{sec:gen}, we comment on generalizations of our results allowing considering types of functionals other than~\eqref{eq:defF}.

\section{Preliminaries}\label{sec:prelim}

We denote by $B_r(x)$ a ball centred in $x$, with radius $r$. 

Given a set $U \subseteq \rn$, $\gamma \in (0, 1]$, and the function $U$, we denote
\begin{equation}
[f]_{0, \gamma} \coloneqq \sup_{x, y \in U, x \neq y} \frac{|f(x) - f(y)|}{|x-y|}\,,
\end{equation}
which is the H\"older seminorm of a function $f$. The set $U$ will be always clear from context.

We say that two real functions $f, g$ are comparable, if there exists a constant $c > 0$ such that $f \leq g \leq cf$. We moreover say that the function $f : [0, \infty) \to [0, \infty)$ satisfies $\Delta_2$-condition if there exists a constant $c > 0$ such that $f(2t) \leq cf(t)$, for any $t$. We denote such situation by $f \in \Delta_2$. 

By $\cH^{n-1}$, we denote classical Hausdorff measure of dimension $n - 1$, defined on $\rn$.\newline

Let us introduce some basic facts concerning spaces of Musielak--Orlicz type~\cite{IC-pocket, C-b, hahab}. With the function $M : \Omega \times \rp \to \R$, defined by \[M(x, t) = t^p + a(x)t^q\qquad \text{for }\ p, q > 1\,,\ \ 0 \leq a \in L^{\infty}\,,\] we can define the corresponding Musielak--Orlicz space by
\begin{equation*}
    L_M(\Omega) = \left\{ \xi : \Omega \to \rn \text{ measurable and such that } \int_{\Omega} M(x, |\xi(x)|)\,dx < \infty \right\}\,
\end{equation*}equipped with the Luxemburg norm
$$
\|\xi\|_{L_{M}(\Omega)}:=\inf \left\{\lambda > 0: \int_{\Omega}M \left(x,\frac{\xi(x)}{\lambda}\right)dx\leq 1 \right\}.
$$
Related Sobolev space $W(\Omega)$, defined in~\eqref{eq:Wdef}, is considered with a norm\[\|v\|_{W(\Omega)}:=\|v\|_{L^1(\Omega)}+\|\nabla v\|_{L_M(\Omega)}\,.\] 

We say that a sequence $(\xi_k)_k$ converges to $\xi$ modularly in $L_M(\Omega)$, if
\begin{equation}
    \int_{\Omega} M(x, |\xi_k - \xi|)\,dx \xrightarrow[]{k \to \infty} 0\,,
\end{equation}
and we denote it by $\xi_k \xrightarrow[k \to \infty]{M} \xi$. We mention the Generalized Vitali Convergence Theorem from the~\cite[Theorem 3.4.4]{C-b}, stating that
\begin{align}\label{eq:genVit}
    \xi_k \to \xi \text{ modularly } \iff &\text{ the family } \{M(\cdot, |\xi_k(\cdot)|)\}_{k} \text{ is uniformly integrable}\nonumber\\
    &\text{ and $(\xi_k)_k$ converges in measure to $\xi$.} 
\end{align}
Let us recall the space $W(\Omega)$ defined in~\eqref{eq:Wdef}. By the choice of $M$, it is equivalent to say that the sequence $(u_k)_k \subset W(\Omega)$ converges to $u \in W(\Omega)$ in the strong topology of $W(\Omega)$ and that
\begin{equation}\label{modular-convergence}
    u_k \xrightarrow[k \to \infty]{L^1} u \text{ in $L^1(\Omega)$} \quad \text{and} \quad \nabla u_k \xrightarrow[k \to \infty]{M} \nabla u \text{ modularly.}
\end{equation}
Let us mention a simple lemma, following from the Lebesgue Dominated Convergence Theorem.
\begin{lem}\label{lem:trunc} If $T_k(x)=\min\{k,\max\{-k,x\}\}$ for $k>0$ and $x\in\R$, $M(x, t) = t^p + a(x)t^q$, $\vp \in W(\Omega)$, then for $k\to\infty$ we have $T_k \vp\to\vp$ in $W(\Omega)$.
\end{lem}

We introduce the approximation method by convolution with shrinking. This method is of use in many papers concerning the absence of the Lavrentiev's phenomenon and density of smooth functions in Musielak--Orlicz--Sobolev spaces, see~\cite{yags, Bor-Chl, ibm, bgs-arma}.

Let us fix $n, m \in \N$ and let $U \subset \rn$ be a bounded star-shaped domain with respect to a ball $B(x_0, R)$. Define $\kappa_{\delta} = 1 - \frac{\delta}{R}$. Moreover, let $\rho_{\delta}$ be a standard regularizing kernel on $\rn$, that is $\rho_{\delta}(x) = \rho(x/\delta)/\delta^n$, where $\rho \in C^\infty(\rn)$, $\mathrm{supp}\,\rho \Subset B(0, 1)$ and $\int_{\rn} \rho (x)\,dx = 1$, $\rho (x) = \rho (-x)$, such that $0\leq \rho\leq 1$. Then for any measurable function $v : \rn \to \mathbb{R}^m$, we define the function $S_{\delta}v : \rn \to \mathbb{R}^m$ by
\begin{equation}\label{Sdxi}
    S_{\delta}v(x) := 
 \int_{U} \rho_\delta( x-y)v\left(x_0 + \frac{y - x_0}{\kappa_{\delta}}\right)\,dy = \int_{B_{\delta}(0)} \rho_{\delta}(y)v\left(x_0 + \frac{x-y-x_0}{\kappa_{\delta}}\right)\,dy\,.
\end{equation}
By direct computations, one can show that $S_{\delta}v$ has a compact support in $U$ for $\delta \in (0, R/4)$. Moreover, we observe that for $v \in W(U)$ it holds that
\begin{equation}\label{eq:nablaout}
    \nabla S_{\delta}v = \tfrac{1}{\kappa_{\delta}} S_{\delta}(\nabla v)\,.
\end{equation}
We introduce other useful properties of this approximation in the following lemmas.
\begin{lem}[Lemma~3.1 in~\cite{ibm}]\label{lem:dec2}
    If $v \in L^1(U)$, then $S_{\delta}v$ converges to $v$ in $L^1(U)$, and so in measure, as $\delta \to 0$. 
\end{lem}
%
%
\begin{lem}[Lemma~3.3 in~\cite{ibm}]\label{lem:dec1}
Let $v \in W_0^{1, 1}(U)$, where ${U}$ is a star-shaped domain with respect to a ball $B(x_0, R)$. It holds that
\begin{itemize}[--]
    \item if $v \in L^{\infty}({U})$, then
    \begin{equation}\label{eq:inqLinf}
        \|\nabla S_{\delta}(v)\|_{L^{\infty}} \leq \delta^{-1}\|v\|_{L^{\infty}}\|\nabla \rho\|_{L^1}\,;
    \end{equation}
    \item if $v \in C^{0, \gamma}({U})$, $\gamma \in (0, 1]$, then
    \begin{equation}\label{eq:inqH}
        \|\nabla S_{\delta}(v)\|_{L^{\infty}} \leq \frac{\delta^{\gamma-1}}{\kappa_{\delta}^\gamma}[v]_{{0, \gamma}}\|\nabla \rho\|_{L^1}\,.
    \end{equation}
\end{itemize}
\end{lem}

\section{Approximation and the absence of the Lavrentiev's phenomenon}\label{sec:approximation-and-no-gap}

When  $q>p>1$ no matter how far are $q$ and $p$, to ensure the approximation properties of the double phase version of the Sobolev space it suffices to control the decay of the weight close to the phase transition. In fact, it is enough to have $a^{\frac{1}{q-p}}$ to be comparable to a Lipschitz continuous function. In Section~\ref{ssec:approx} we prove the density result, which is applied in Section~\ref{ssec:no-gap} to get the absence of the Lavrentiev's phenomenon. 

\subsection{Approximation}\label{ssec:approx}
In this section, we make use of the convolution with shrinking, to establish the density of smooth functions in the energy space $W$ defined in~\eqref{eq:Wdef}. The result reads as follows.
\begin{theo}[Density of smooth functions]\label{theo:approx}
Let $\Omega$ be a bounded Lipschitz domain in $\rn$, 
$1<p<q<\infty$, $\vk>0$, and $a:\Omega\to\rp$ be such that $a \in \SP^{\vk}(\Omega)$. Then the following assertions hold true.
\begin{enumerate}[{\it (i)}]
\item  If  $\vk \geq q-p$, then for any $\vp\in W(\Omega)$  there exists a sequence $(\vp_\delta)_{\delta}\subset C_c^\infty(\Omega)$, such that  $\vp_\delta\to  \vp$ in $W(\Omega)$.
\item  Let $\gamma \in (0, 1]$. If $\vk \geq (q-p)(1-\gamma)$, then for any $\vp\in W(\Omega)\cap C^{0,\gamma}(\Omega)$  there exists a sequence $(\vp_\delta)_{\delta}\subset C_c^\infty(\Omega)$, such that  $ \vp_\delta\to  \vp$  in $W(\Omega)$.
\end{enumerate} 
Moreover, in both above cases, if $\vp\in L^\infty(\Omega)$, then there exists $c=c(\Omega)>0$, such that $\|\vp_\delta\|_{L^\infty(\Omega)}\leq c \|\vp \|_{L^\infty(\Omega)}$.
\end{theo}

\begin{proof}
    Let us at first notice that by Lemma~\ref{lem:trunc}, we have that $W(\Omega) \cap L^{\infty}(\Omega)$ is dense in $W(\Omega)$. Therefore, for the assertion $(i)$, it suffices to consider the density of $C_c^{\infty}(\Omega)$ in $W(\Omega) \cap L^{\infty}(\Omega)$. Let us assume that in case of $(i)$, we have $\gamma = 0$. We shall prove the claims $(i)$ and $(ii)$ simultaneously. To this aim, let us take any $\vp \in W(\Omega) \cap L^{\infty}(\Omega)$ in the case of $\gamma = 0$ and $\vp \in W(\Omega) \cap C^{0, \gamma}(\Omega)$ otherwise. \newline

    At first, let us assume that $\Omega$ is a star-shaped domain with respect to a ball centred in zero and with radius $R > 0$, that is $B(0, R)$. Recall the definition of $S_{\delta}\vp$, given in~\eqref{Sdxi}, where we take $U = \Omega$, $x_0 = 0$, and $\delta < R/4$. Our aim now is to prove that $\nabla S_{\delta}\vp$ converges to $\nabla \vp$ in $W(\Omega)$. {Due to~\eqref{modular-convergence}, it is enough to show that   $S_{\delta}\vp\to \vp$ in $L^1$ and $\nabla S_{\delta}\vp\xrightarrow{M}\nabla \vp$ modularly in $L_M$. We observe that by~\eqref{eq:nablaout} and Lemma~\ref{lem:dec2}, we have this first convergence as well as the fact that $\nabla (S_{\delta}(\vp))$ converges to $\vp$ in measure.} Therefore, by~\eqref{eq:genVit}, it suffices to prove that
    \begin{equation}\label{eq:jan6}
        \text{the family } \big\{ |\nabla(S_{\delta} \vp)(\cdot))| ^p+a(\cdot) |\nabla(S_{\delta} \vp)(\cdot))|^q \big\}_{\delta} \text{ is uniformly integrable.}
    \end{equation}
    Observe that by Lemma~\ref{lem:dec1}, for sufficiently small $\delta > 0$, there exist a constant $C_S > 0$, independent of $\delta$, such that
    \begin{equation}\label{eq:jan1}
    \|\nabla(S_{\delta}\vp)\|_{L^{\infty}} \leq C_S\delta^{\gamma - 1}\,.    \end{equation} Indeed, if $\gamma = 0$, then by using assertion~\eqref{eq:inqLinf} and the fact that $\vp \in L^{\infty}(\Omega)$, we can set $C_S \coloneqq \|\vp\|_{L^{\infty}}\|\nabla \rho\|_{L^1}$ in~\eqref{eq:jan1}. In the case of $\gamma \in (0, 1]$,~\eqref{eq:inqH} provides that $\|\nabla S_{\delta}(\vp)\|_{L^{\infty}} \leq \frac{\delta^{\gamma-1}}{\kappa_{\delta}^\gamma}[\vp]_{{0, \gamma}}\|\nabla \rho\|_{L^1}$. As $\vp \in C^{0, \gamma}(\Omega)$ and $\kappa_{\delta} \xrightarrow{\delta \to 0} 1$, we obtain inequality~\eqref{eq:jan1} with constant $C_S \coloneqq 2[\vp]_{0, \gamma}\|\nabla \rho\|_{L^1}$ for sufficiently small $\delta$. We therefore have~\eqref{eq:jan1} for all $\gamma \in [0,1]$. \newline
    
    As $a \in \SP^{\vk}$, there exists a constant $C_a > 1$ such that for any $x, y \in \Omega$ we have $a(x) \leq C_a(a(y) + |x-y|^{\vk})$. Let us take any $x, y \in \Omega, \tau > 0, \delta \in (0, 1)$ such that $|x - y| \leq \tau \delta$. We have
    \begin{align}\label{eq:jan7}
        |\nabla S_{\delta}(\vp)(x)|^p + a(x)|\nabla S_{\delta}(\vp)(x)|^q &= |\nabla S_{\delta}(\vp)(x)|^p(1 + a(x)|\nabla S_{\delta}(\vp)(x)|^{q-p})\nonumber\\
        &\leq |\nabla S_{\delta}(\vp)(x)|^p(1 + C_a(a(y) + \tau^{\vk}\delta^{\vk})|\nabla S_{\delta}(\vp)(x)|^{q-p})\nonumber\\
        &\leq C_a|\nabla S_{\delta}(\vp)(x)|^p(1 + a(y)|\nabla S_{\delta}(\vp)(x)|^{q-p} + \tau^{\vk}\delta^{\vk}|\nabla S_{\delta}(\vp)(x)|^{q-p})\,.
    \end{align}
    By the inequality~\eqref{eq:jan1}, we obtain that
    \begin{equation}\label{eq:jan8}
        \delta^{\vk}|\nabla S_{\delta}(\vp)(x)|^{q-p} \leq C_S^{q-p}\delta^{\vk}\delta^{(q-p)(\gamma - 1)} \leq C_S^{q-p}\,,
    \end{equation} 
    where in the last inequality we used that $\delta \in (0, 1)$ and $\vk + (q-p)(\gamma - 1) \geq 0$. By~\eqref{eq:jan7} and~\eqref{eq:jan8}, we have that there exists a constant $C_\tau > 0$, not depending on $\delta$, such that
    \begin{align}\label{eq:jan2}
      |\nabla S_{\delta}(\vp)(x)|^p + a(x)|\nabla S_{\delta}(\vp)(x)|^q 
        \leq C_\tau\left(|\nabla S_{\delta}(\vp)(x)|^p + \left(\inf_{z \in B_{\tau\delta}(x)}a(z)\right)|\nabla S_{\delta}(\vp)(x)|^q\right)\,.
    \end{align}
    Let us recall~\eqref{eq:nablaout}, that is $\nabla S_{\delta}(\vp) = \tfrac{1}{\kappa_{\delta}}S_{\delta}(\nabla \vp)$. By using Jensen's inequality in conjunction with the fact that $\kappa_{\delta} \geq 1/2$ for sufficiently small $\delta$, we may write
    \begin{align}\label{eq:jan3}
        |\nabla S_{\delta}(\vp)(x)|^p &= \frac{1}{\kappa_{\delta}^p}\left| \int_{B_{\delta}(0)}\rho_{\delta}(y)(\nabla \vp)((x-y)/\kappa_{\delta})\,dy \right|^p\nonumber \\ &\leq 2^p\int_{B_{\delta}(0)} \rho_{\delta}(y)|(\nabla \vp)((x-y)/\kappa_{\delta})|^p\,dy = 2^pS_{\delta}(|\nabla \vp(\cdot)|^p)(x)\,
    \end{align}
    for sufficiently small $\delta > 0$. Analogously, it holds that
    \begin{align}\label{eq:jan4}
        \left(\inf_{z \in B_{\tau\delta}(x)}a(z)\right)|\nabla S_{\delta}(\vp)(x)|^q &\leq 2^q\int_{B_{\delta}(0)}\rho_{\delta}(y)\left(\inf_{z \in B_{\tau\delta}(x)}a(z)\right)|(\nabla \vp)((x-y)/\kappa_{\delta})|^q \,dy\nonumber\\
        &\leq 2^q\int_{B_{\delta}(0)}\rho_{\delta}(y)a((x-y)/\kappa_{\delta})|(\nabla \vp)((x-y)/\kappa_{\delta})|^q \,dy= 2^qS_{\delta}(a(\cdot)|\nabla \vp(\cdot)|^q)(x)\,,
    \end{align}
    where $\tau$ is fixed such that for sufficiently small $\delta > 0$ we have
    \begin{equation*}
        \left| (x-y)/\kappa_{\delta} - x \right| \leq \frac{|y|}{\kappa_{\delta}} + \frac{1-\kappa_{\delta}}{\kappa_{\delta}}|x| \leq \frac{\delta}{\kappa_{\delta}} + \frac{\delta}{2R\kappa_{\delta}}(\text{diam }\Omega) \leq \tau\delta\,.
    \end{equation*}
        Observe that by~\eqref{eq:jan2} and by estimates~\eqref{eq:jan3} and~\eqref{eq:jan4}, we have
    \begin{equation}\label{eq:jan12}
        M(x, |\nabla S_{\delta}\vp(x)|) \leq 2^qC_{\tau}(S_{\delta}(|\nabla \vp(\cdot)|^p)(x) + S_{\delta}(a(\cdot)|\nabla \vp(\cdot)|^q)(x)) = 2^qC_{\tau}S_{\delta}\left(M\left(\cdot, |\nabla \vp(\cdot)|\right) \right)(x)\,.
    \end{equation}
    The fact that $\vp \in W(\Omega)$ implies that $M(\cdot, |\nabla \vp(\cdot)|) \in L^1(\Omega)$. Therefore, Lemma~\ref{lem:dec2} gives us that the sequence $\left(S_{\delta}\left(M\left(\cdot, |\nabla \vp(\cdot)|\right) \right)\right)_\delta$ converges in $L^1$. By the Vitali Convergence Theorem, it means that the family $\left\{S_{\delta}\left(M\left(\cdot, |\nabla \vp(\cdot)|\right) \right)\right\}_{\delta}$ is uniformly integrable. Using the estimate~\eqref{eq:jan12}, we deduce that the family $\{M(\cdot, |\nabla(S_{\delta}\vp)(\cdot)|)\}_{\delta }$ is uniformly integrable, which is~\eqref{eq:jan6}. Therefore, the proof is completed for $\Omega$ being a bounded star-shaped domain with respect to a ball centred in zero. \newline

    To prove the result for $\Omega$ being star-shaped with respect to a ball centred in point other than zero, one may translate the problem, obtaining the set being a star-shaped domain with respect to a ball centred in zero. Then, proceeding with the proof above and reversing translation of $\Omega$ gives the desired result. \newline
    
    Now we shall focus on the case of $\Omega$ being an arbitrary bounded Lipschitz domain. By \cite[Lemma~8.2]{C-b}, a set $\overline{\Omega}$ can be covered by a finite family of sets $\{U_i\}_{i=1}^{K}$ such that each $\Omega_i :=\Omega\cap U_i$ is a star-shaped domain with respect to some ball. Then $\Omega=\bigcup_{i=1}^{K}\Omega_i\,$.    By~\cite[Proposition 2.3, Chapter 1]{necas}, there exists the partition of unity related to the partition $\{U_i\}_{i=1}^{K}$, i.e., the family $\{\theta_i\}_{i=1}^{K}$ such that
\begin{equation*}
0\le\theta_i\le 1,\quad\theta_i\in C^\infty_c(U_i),\quad \sum_{i=1}^{K}\theta_i(x)=1\ \ \text{for}\ \ x\in\Omega\,.
\end{equation*}
By the previous paragraph for every $i = 1, 2, \dots, K$, as $\Omega_i$ is a star-shaped domain with respect to some ball, and $\theta_i\vp \in W(\Omega_i)$, there exist a sequence $(\vp_{\delta}^i)_{\delta}$ such that $\vp_{\delta}^i \xrightarrow[]{\delta \to 0} \theta_i \vp$ in $W(\Omega_i)$. Let us now consider the sequence $(I_\delta)_\delta$ defined as
\[I_{\delta} \coloneqq \sum_{i=1}^{K} \vp^i_{\delta}.\] 
We shall show that $I_{\delta} \to \vp$ in $W(\Omega)$. As we have that $\vp_{\delta}^i \to \theta_i\vp$ in $L^1$ for every $i$, we have $I_{\delta} \to \vp$ in $L^1$. It suffices to prove that $\nabla I_{\delta} \to \nabla \vp$ in $L_M(\Omega)$. Since the sequence $(\nabla  \vp_{\delta}^i)_\delta$ converges to $\nabla (\theta_i \vp)$ in measure and $\sum_{i=1}^{K} \nabla(\theta_i \vp) = \nabla \vp$, it holds that
\begin{equation}\label{eq:jan9}
    \left(\nabla I_{\delta}\right)_{\delta} 
    \to  \nabla \vp \text{ in measure.}
\end{equation}
Moreover, for any $x \in \Omega$ we have that
\begin{align}\label{eq:jan11}
    \left| \nabla I_{\delta}(x) \right|^p + a(x)\left| \nabla I_{\delta}(x) \right|^q &\leq \sum_{i=1}^{K} \left(K^{p-1}|\nabla(\vp_{\delta}^i)(x)|^p + K^{q-1}a(x)|\nabla(\vp_{\delta}^i)(x)|^q\right)\nonumber\\
    &\leq K^{q-1}\sum_{i=1}^{K}  \left(|\nabla(\vp_{\delta}^i)(x)|^p + a(x)|\nabla(\vp_{\delta}^i)(x)|^q\right)\,.
\end{align}
As for all $i = 1, 2, \dots, K$, we have that $(\vp^i_{\delta})_{\delta}$ converges in $W(\Omega_i)$, it holds that the family\\ $\{|\nabla(\vp_{\delta}^i)(\cdot)|^p + a(\cdot)|\nabla(\vp_{\delta}^i)(\cdot)|^q\}_{\delta}$ 
is uniformly integrable. Therefore, the estimate~\eqref{eq:jan11} gives us that\[\text{ the family $\quad\left\{ \left| \sum_{i=1}^{K} \nabla(\vp_{\delta}^i)(\cdot)\right|^p + a(\cdot) \left| \sum_{i=1}^{K} \nabla(\vp_{\delta}^i)(\cdot)\right|^q\right \}_{\delta}\quad$ is uniformly integrable. } \]
This together with~\eqref{eq:jan9} and~\eqref{eq:genVit}, as well as the fact that  $I_{\delta} \to \vp$ in $L^1$, gives us the result for an arbitrary bounded Lipschitz domain $\Omega$.
\end{proof}

\subsection{Absence of the Lavrentiev's phenomenon}\label{ssec:no-gap}
As a direct consequence of Theorem~\ref{theo:approx} we infer the absence of the Lavrentiev's phenomenon. We start with a simple formulation for a double phase functional~\eqref{eq:defF} reading as follows. 
\begin{theo}[Absence of the Lavrentiev's phenomenon for a model functional]\label{theo:lavrentiev-model}
 Suppose  $\Omega\subset\rn$ is  a bounded Lipschitz domain, $1<p < q<\infty$, $\vk>0$, $a:\Omega\to\rp$, $\cF$ is given by \eqref{eq:defF}, and $W(\Omega)$ is defined in~\eqref{eq:Wdef}.  Assume that $u_0$ satisfies $\cF[u_0]<\infty$ and $a \in \SP^{\vk}(\Omega)$.  Then the following assertions hold true.
    \begin{enumerate}
        \item[(i)] If  $\vk \geq q-p$, then
            \begin{equation}
        \inf_{v \in u_0 + W(\Omega)} \mathcal{F}[v] = \inf_{w \in u_0 + C_c^\infty(\Omega)} \mathcal{F}[w]\,. 
    \end{equation}
        \item[(ii)] Let $\gamma \in (0, 1]$. If $\vk \geq (q-p)(1-\gamma)$, then
            \begin{equation}
        \inf_{v \in u_0 + W(\Omega) \cap C^{0, \gamma}(\Omega)} \mathcal{F}[v] = \inf_{w \in u_0 + C_c^\infty(\Omega)} \mathcal{F}[w]\,.
    \end{equation}
    \end{enumerate}
\end{theo}
The above theorem is a special case of the following more general result. Let us consider the following variational functional 
\begin{align}\label{cG-def}
    \cG[u]:=\int_\Omega G(x,u,\nabla u)\,dx\,,
\end{align}
over an open and bounded set $\Omega\subset\rn$, $n\geq 1$, where $G:\Omega\times\R\times\rn\to\R$ is merely continuous with respect to the second and the third variable. We suppose that there exist constants $0<\nu<1<L$ and a nonnegative $h\in L^1(\Omega)$ such that
\begin{align}\label{sandwich}
    \nu \left(|\xi|^p+a(x)|\xi|^q\right)\leq G(x,z,\xi)\leq L\left(|\xi|^p+a(x)|\xi|^q+\Lambda(x)\right), \qquad\text{for all }\ x\in\Omega,\ z\in\R,\ \xi\in\rn\,. 
\end{align}

\begin{theo}[Absence of Lavrentiev's phenomenon for general functionals]\label{theo:lavrentiev-general}
 Suppose  $\Omega\subset\rn$ is  a bounded Lipschitz domain, $1<p< q<\infty$, $\vk>0$, $a:\Omega\to\rp$, $\cG$ is given by \eqref{cG-def}, and $W(\Omega)$ is defined in~\eqref{eq:Wdef}. Assume that {$u_0$ satisfies $\cG[u_0]<\infty$ and} $a \in \SP^{\vk}(\Omega)$. Then the following assertions hold true.
    \begin{enumerate}
        \item[(i)] If  $\vk \geq q-p$, then
            \begin{equation}\label{no-gap-gen}
        \inf_{v \in u_0 + W(\Omega)} \cG[v] = \inf_{w \in u_0 + C_c^\infty(\Omega)} \cG[w]\,. 
    \end{equation}
        \item[(ii)] Let $\gamma \in (0, 1]$. If $\vk \geq (q-p)(1-\gamma)$, then
            \begin{equation}\label{no-gap-hold}
        \inf_{v \in u_0 + W(\Omega) \cap C^{0, \gamma}(\Omega)} \cG[v] = \inf_{w \in u_0 + C_c^\infty(\Omega)} \cG[w]\,.
    \end{equation}
    \end{enumerate}
\end{theo}

\begin{proof} Since $C_c^\infty(\Omega)\subset 
W(\Omega)$, it holds that $\inf_{v \in u_0+W(\Omega)}\cG[v]\leq \inf_{w \in u_0+C_c^\infty(\Omega)}\cG[w]\,.$ Let us concentrate on showing the opposite inequality. By direct methods of calculus of variation, there exists a minimizer, i.e., a function $u \in W(\Omega)$ such that
\begin{equation*}
    \cG[u_0 + u] = \inf_{v \in u_0+W(\Omega)}\cG[v]\,.
\end{equation*}
By assertion $(i)$ of Theorem~\ref{theo:approx}, there exists $(u_k)_{k}\subset C_c^\infty(\Omega)$ such that $u_k\to u$ in $W(\Omega)$. Since $G$ is continuous with respect to the second and the third variable, we infer that \[
\text{$G(x,u_0(x)+u_k(x), \nabla u_0(x)+\nabla u_k(x))\xrightarrow[k\to\infty]{} G(x,u(x),\nabla u(x))$ in measure.}
\] We shall now show that
\begin{equation}
    \label{unif-int}\text{the family $\big\{G(x,u_0(x)+u_k(x),\nabla u_0(x)+\nabla u_k(x))\big\}_{k\geq1}$ {is uniformly integrable}.}\end{equation}
By assumption \eqref{sandwich}, we notice that
\begin{align*}
G(x,u_0(x)+u_k(x),\nabla u_0(x) +&\nabla u_k(x)) \leq L\left(| \nabla u_0(x)+\nabla u_k(x)|^p+a(x)|\nabla u_0(x)+\nabla u_k(x)|^q\right) +L\Lambda(x)\\
&\leq C\left(|\nabla u_k(x)|^p+a(x)|\nabla u_k(x)|^q\right)  + C\left(|\nabla u_0(x)|^p+a(x)|\nabla u_0(x)|^q\right) +L\Lambda(x)\,,
\end{align*}
where $C$ is a positive constant, for every fixed $k\geq 1$. Note that  $h\in L^1(\Omega)$ and as $\cG[u_0] < \infty$, by~\eqref{sandwich}, we have $\int_{\Omega}|\nabla u_0(x)|^p+a(x)|\nabla u_0(x)|^q\, dx<\infty$. Moreover, since $(\nabla u_k)_{k}$ converges in $W(\Omega)$, we infer that\[\text{the family}\quad  \{|\nabla u_k(x)|^p+a(x)|\nabla u_k(x)|^q\}_{k}\quad \text{is uniformly integrable.}\] Thus, \eqref{unif-int} is justified. In turn, by Vitali Convergence Theorem, we have that $\cG[u_0+u_k] \xrightarrow[]{k \to \infty} \cG[u + u_0]$. Therefore, we have that
\begin{equation*}
    \inf_{w \in u_0 + C_c^\infty(\Omega)} \cG[w] \leq \cG[u_0 + u] = \inf_{v \in u_0 + W(\Omega)} \cG[v]
\end{equation*}
Consequently, \eqref{no-gap-gen} is proven.

By repeating the same procedure for $u\in W(\Omega)\cap C^{0,\gamma}(\Omega)$ with the use of Theorem~\ref{theo:approx} {\it (ii)} instead of {\it (i)}, one gets~\eqref{no-gap-hold}. 
\end{proof}

\section{Sharpness}\label{sec:example}
%
Let us recall the energy space $W(\Omega)$ given by~\eqref{eq:Wdef}. We state in Theorem~\ref{theo:main} that the range for the absence of the Lavrentiev's phenomenon is sharp. By sharpness, we mean that if $p,q$ and $\vk$ are outside the proper range~\eqref{range-no-gamma}, it is possible to find a Lipschitz domain $\Omega$, a weight $a \in \SP^{\vk}(\Omega)$ and a boundary data $u_0 \in W(\Omega)$ such that the Lavrentiev's phenomenon occurs. In what follows we consider the double-phase functional $\cF$ defined in~\eqref{eq:defF}, that is
 \begin{equation}\label{eq:defFB}
         \cF[u] = \int_{\Omega}  |\nabla u(x)|^p + a(x)|\nabla u(x)|^{q}\,dx\,.
 \end{equation}
 With Theorem~\ref{theo:sharpness} our aim is to show the occurrence of the Lavrentiev's phenomenon between the spaces $ u_0 + W(\Omega)$ and $ u_0 + C_c^\infty(\Omega)$, whenever $p, q$ and $ \vk$ are as in Theorem \ref{theo:main}, $(ii)$. We point out that, for our example, we modify the  construction from~\cite{eslemi} based on the  seminal idea of Zhikov's checkerboard~\cite{zh-86,zh}.

\begin{theo}[Sharpness]
    \label{theo:sharpness}
    Let $\cF$ be defined by~\eqref{eq:defF} and $p, q, \vk > 0$ such that
    $$1<p<n<n+\vk<q\,.$$ Then there exist a Lipschitz domain $\Omega$, a function $a \in \SP^{\vk}$ and $u_0 \in W(\Omega)$ satisfying $\mathcal{F}[u_0]<\infty$, such that 
    \begin{equation}
    \inf_{v \in u_0 + W(\Omega)} \cF[v] < \inf_{w \in u_0 + C_c^\infty(\Omega)}\cF[w]\,. \label{lavrentiev}
\end{equation}
\end{theo}
In order to show the presence of the Lavrentiev's phenomenon we first define the Lipschitz domain $\Omega$, the function $a$ and the boundary data $u_0$. We choose $\Omega$ as the ball of centre $0$ and radius $1$, i.e.,
\begin{equation} \label{domain}
    \Omega = B_1 \coloneqq B_1(0)\,.
\end{equation}
Now let us define the following set
\begin{equation}\label{eq:defCp}
    V:= \left \{x \in B_1:\ x_n^2 - \sum_{i=1}^{n-1}x_i^2 > 0 \right \} \, .
\end{equation}
\begin{figure}[htbp]\label{fig:sharpness}
  \centering
  \includegraphics[width=6cm,height=6cm]{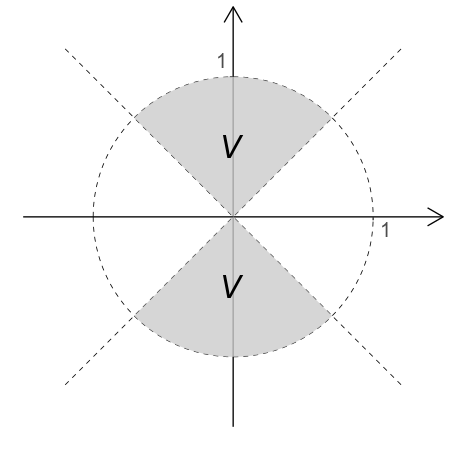}
  \caption{The main properties of $a\in\SP^\vk(B_1)$ and $u_*\in W(B_1)$ that produce the Lavrentiev's gap for our counterexample are the facts that $\supp a\subset V$, and $\nabla u_*\equiv 0$ in $B_1\setminus V$.}
\end{figure}
Regarding the weight $a$ we introduce the  function $\ell: \mathbb{R}^n \to \mathbb{R}$ via the following formula
\begin{equation*}
  \ell(x) := \max \left \{x_n^2 - \sum_{i=1}^{n-1}x_i^2, 0 \right \}|x|^{-1}\,,\qquad x = (x_1, \dots, x_n)\,.
\end{equation*}
The weight is defined as
\begin{equation} \label{the-weight}
    a := \ell^\vk\,. 
\end{equation}
Computing the partial derivative of $\ell$ in $V$ we get
\begin{equation*}
 \frac{\partial \ell}{\partial x_i} = \begin{cases}    -\tfrac{x_i}{|x|^3}{\textstyle \left(\sum_{j=1}^{n-1}x_j^2 +3x_n^2\right)} & \text{if}  \quad i=1,\dots,n-1 \,,  \displaybreak[1] \\ \noalign{\vskip4pt}
  \; \; \; \tfrac{x_i}{|x|^3}{\textstyle 
      \left(\sum_{j=1}^{n-1}3x_j^2 + x_i^2 \right)}  & \text{if} \quad i=n\,.
\end{cases}
\end{equation*}
We can observe that $\| \nabla \ell \|_{L^{\infty}(B_1)}$ is bounded. In turn, $\ell$ is Lipschitz continuous and consequently $a \in \SP^\vk(B_1)$. {We note that $\supp a\subset V$, so the set $V$ shall include whole $p$-$q$-phase, while $p$-phase will be in $B_1\setminus V$.}

Let us state and prove a lemma that we will use in the proof of Theorem~\ref{theo:sharpness}.
\begin{lem}  \label{lemma_1}
Let $a$ be defined by~\eqref{the-weight} and $V$ be defined as in~\eqref{eq:defCp}. Then
\begin{equation}
   r_1 \coloneqq \int_{V} {|x|^{-\frac{q(n-1)}{q-1}}}{a(x)^{-\frac{1}{q-1}}} \, dx\, < \infty\,. \label{integral}
\end{equation}
\end{lem}

\begin{proof}
We use the spherical coordinates. {The proof is presented in two cases -- for $n=2$ and $n>2$.} \newline 

For $n=2$ we take
  $$x_1:= \rho \cos\theta\quad\text{and} \quad x_2:= \rho \sin\theta\,,$$
consequently
  $$a=\rho^{\vk}\max(-\cos 2\theta,0)^{\vk},$$
  where $\theta \in [0,2\pi)$. After this change of variables $V$ is mapped into $S \coloneqq (0, 1) \times  \left [(\frac{\pi}{4}, \frac{3\pi}{4}) \cup (\frac{5\pi}{4}, \frac{7\pi}{4} ) \right ]$, so 
~\eqref{integral} reads as
\begin{equation*}
 {r_1=}\int_{S} \rho^{1 - \frac{q + \vk}{q-1}}\left|\cos{(2\theta)} \right|^{-\frac{\vk }{q-1}}\,d\rho\,d\theta\,.
\end{equation*}
As $q > \vk + 2$, we have $1 - \frac{q+\vk}{q-1} > -1$, which implies that $\int_{0}^{1} \rho^{1 - \frac{q(1+\vk)}{q-1}}\,d\rho < \infty$.
As far as $-\cos{(2\theta)}^{-\frac{\vk }{q-1}}$ is concerned, we observe at first that over the set that we integrate on, it holds that $\cos(2\theta) = 0$ only for $\theta = \tfrac{\pi}{4}, \tfrac{3\pi}{4}, \tfrac{5\pi}{4}, \tfrac{7\pi}{4}$. Therefore, it suffices to prove integrability of $\left|\cos{(2\theta)}\right|^{-\frac{\vk }{q-1}}$ near these points. Observe that for sufficiently small $\theta_0 > 0$ we have
$$\left|\cos \left(2(\theta_0 + \tfrac{\pi}{4}) \right)\right| \geq 2\theta_0 \left (1-\frac{2\theta_0}{\pi} \right ) \geq \theta_0\,,$$
which means that for $\theta = \theta_0 + \tfrac{\pi}{4}$ we get
\begin{equation}\label{eq:gr1}
\left| \cos \left ( 2\theta\right )\right |^{-\frac{\vk }{q-1}} \leq \left ( \theta - \frac{\pi}{4}\right )^{-\frac{\vk }{q-1}}\,.
\end{equation}
Since $q > \vk+1$, we have $-\frac{\vk}{q-1}>-1$, and therefore, we have the integrability of $|\cos{(2\theta)}|^{-\frac{\vk}{q-1}}$ near $\tfrac{\pi}{4}$, and by analogy, also in points $\tfrac{3\pi}{4}, \tfrac{5\pi}{4}, \tfrac{7\pi}{4}$. Therefore, we showed that $r_1$ is finite for $n = 2$. \newline

For $n > 2$ we set
$$x_1 := \rho \cos{\theta}\prod_{k=1}^{n-2}\sin{\theta_{k}}\,, \quad x_2 := \rho \sin{\theta}\prod_{k=1}^{n-2}\sin{\theta_{k}}\,,\quad x_i := \rho\cos{\theta_{n-2}}\prod_{k=i-1}^{n-2}\sin{\theta_k}\,, \text{ for } i \geq 3\, $$
and so
$$a=\rho^{\vk} \max(\cos2\theta_{n-2},0)^\vk\,,$$ with $\rho>0$ and $\theta_i \in [0,\pi]$ for $i = 1, \dots, n-2$. We observe that $V$ is mapped to $S = (0, 1) \times (0, 2\pi) \times (0, \pi)^{n-2} \times \left(\left(0, \tfrac{\pi}{4}\right) \cap \left(\tfrac{3\pi}{4}, \pi\right)\right)$, that is, $\theta_{n-2} \in \left(0, \tfrac{\pi}{4}\right) \cap \left(\tfrac{3\pi}{4}, \pi\right)$ and the modulus of the determinant of changing the variables may be estimated by $\rho^{n-1}$. Therefore, we can estimate
\begin{equation*}
    r_1 \leq \int_{S} \rho^{n-1 + \frac{q(1-n) -\vk}{q-1}}\left|\cos(2\theta_{n-2})\right|^{-\frac{\vk }{q-1}}\,d\rho\,{d\theta_{n-2}}\,.
\end{equation*}
As $q > \vk+n$, it follows that $n - 1 + \frac{q(1-n) -  \vk}{q-1} > -1$, and therefore, $\int_{0}^{1} \rho^{n - 1 + \frac{q(1-n) - \vk}{q-1}}\,d\rho < \infty$. Using analogous estimates as~\eqref{eq:gr1}, one may also prove the integrability of $\left(\cos(2\theta_{n-2})\right)^{-\frac{\vk }{q-1}}$ in $\left(0, \tfrac{\pi}{4}\right) \cap \left(\tfrac{3\pi}{4}, \pi\right)$, obtaining the finiteness of $r_1$ in case of $n \geq 3$. 
\end{proof}
As far as the boundary data is concerned we first define a function $u_*$ and, after we establish some of its properties, we shall find $u_0$ such that $u_* \in (u_0 + W(B_1))$, but $u_*\not\in \overline{(u_0 + C_c^\infty(B_1))}^W$. We set
\begin{equation} \label{u_star_n_2}
u_*(x) \coloneqq \begin{cases} \displaystyle \sin(2\theta) & \text{if}  \quad 0 \leq \theta \leq \frac{\pi}{4}\,,  \displaybreak[1] \\
\displaystyle  1  & \text{if} \quad \frac{\pi}{4} \leq \theta \leq \frac{3\pi}{4}\,, \displaybreak[1] \\
\displaystyle \sin(2\theta - \pi) & \text{if}  \quad \frac{3\pi}{4}\leq \theta \leq \frac{5\pi}{4}\,, \displaybreak[1] \\
\displaystyle  -1  & \text{if} \quad \frac{5\pi}{4} \leq \theta \leq \frac{7\pi}{4}\,, \displaybreak[1] \\
\displaystyle \sin(2\theta) & \text{if}  \quad \frac{7\pi}{4}\leq \theta \leq 2\pi
\end{cases}
\end{equation}
for $n=2$, and as
\begin{equation} \label{u_star}
u_*(x) \coloneqq \begin{cases} \displaystyle 1 & \text{if}  \quad 0 \leq \vt_{n-2} \leq \frac{\pi}{4}\,,  \displaybreak[1] \\
\displaystyle  \sin(2\vt_{n-2})  & \text{if} \quad \frac{\pi}{4} \leq \vt_{n-2} \leq \frac{3\pi}{4}\,, \displaybreak[1] \\
\displaystyle -1 & \text{if}  \quad \frac{3\pi}{4}\leq \vt_{n-2} \leq \pi\,,
\end{cases}
\end{equation}
for $n \geq 3$. The boundary data $u_0$ is determined by the following expression 
\begin{equation} \label{boundary_data}
    u_0(x):=t_0|x|^2u_*(x)\,,
\end{equation}
where $t_0$ will be chosen. We have the following lemma.
\begin{lem}\label{lemma_2}
The function $ u_*$ belongs to $u_0+W(B_1)$. In particular
    $$r_2
 \coloneqq \int_{B_1}|\nabla u_*(x)|^p \, dx < \infty\,.$$
\end{lem}
\begin{proof}
    We start observing that $\supp a\subset V$ and $\nabla u_*\equiv 0$ in  $\supp a$, i.e.,
    \begin{equation*}
     \int_{B_1}|\nabla u_*(x)|^p +  a(x)|\nabla u_*(x)|^{q}\,dx =  \int_{B_1}|\nabla u_*(x)|^p \,dx =r_2\,.
    \end{equation*}
To justify that $r_2$ is finite, we notice that using spherical coordinates for $n=2$ one gets
$$r_2= \int_0^1 \rho \, d\rho \left [ \int_{0}^{\frac{\pi}{4}} |2\cos(2\theta)|^p \, d\theta + \int_{\frac{3\pi}{4}}^{\frac{5\pi}{4}}|2\cos(2\theta - \pi)|^p \, d\theta + \int_{\frac{7\pi}{4}}^{2\pi} |2\cos(2\theta)|^p \, d\theta \right ]<\infty\,,$$
whereas when $n>2$, then
\[r_2= \int_0^1 \int_0^{2\pi}\int_0^{\pi} |\det J| \, d\rho\, d\theta \prod_{i=1}^{n-3}\,d\theta_i \int_{0}^{\pi}|2\cos(2\theta_{n-2}-\pi)|^p \,d\theta_{n-2}<\infty\,,\]
where $J$ is the Jacobian matrix of the spherical coordinate transformation. Now, since $p<n$ we can apply the Sobolev embedding theorem to obtain $u_* \in L^p(B_1)$. Then $u_* \in W^{1,1}(B_1)$ and $\cF[u_*]  < \infty$, namely $u_* \in W(B_1)$.
\end{proof}

We take 
\begin{equation}
    \label{t}t_0 >\left [ r_2\left ( \frac{q}{r_3} \right )^q \left ( \frac{r_1}{q-1} \right )^{q-1}\right ]^{\frac{1}{q-p}}\,,
\end{equation}
with $r_1$ from Lemma~\ref{lemma_1}, $r_2$ from Lemma~\ref{lemma_2}, and \begin{equation}
    \label{r3}
r_3 \coloneqq \cH^{n-1}(\overline{V} \cap \partial B_1)\,.
\end{equation} 

 Now let  us state the following observation made in~\cite{eslemi}. The proof consists of calculations with the spherical coordinates in which Fubini’s theorem and Jensen’s inequality are used, see~\cite[p.~17]{eslemi} for details.
 
\begin{lem} \label{lemma_3}
For  {any} function $w \in u_0 + C^{\infty}_0(B_1)$ it holds
\begin{equation*}
     t_0
     {\cH^{n-1}(\overline{V} \cap \partial B_1)}  \leq \int_{V} \frac{1
     }{|x|^{n-1}} \left | \left \langle \frac{x}{|x|},\nabla w(x)  \right \rangle\right | \,dx\,,
\end{equation*}
for $t_0$ as in \eqref{t}  {and $u_0$ as in~\eqref{boundary_data}}.
\end{lem}
Now we are ready to prove the theorem.
\begin{proof}[Proof of Theorem~\ref{theo:sharpness}]

Bearing in mind the definition of $u_*$ in \eqref{u_star_n_2}-\eqref{u_star} and of  $u_0$ in \eqref{boundary_data} we start observing that
\begin{align}
  \inf_{v \in u_0 + W(B_1)} \cF[v]  \leq \cF[t_0u_*] \notag &= t_0^p \int_{B_1}|\nabla u_*(x)|^p \,dx + t_0^q \int_{B_1} a(x)|\nabla u_*(x)|^{q}\,dx \notag \\ & =  t_0^p \int_{B_1}|\nabla u_*(x)|^p \,dx  = t_0^p r_2\,, \label{estimate on the minimum}
\end{align}
which is finite by Lemma~\ref{lemma_2}. {Let us fix arbitrary $w\in u_0+C_0^\infty(B_1)$ and $\lambda>0$. In order to estimate from below $\cF[w]$ we notice that}  Lemma \ref{lemma_3} together with Young's inequality Lemma~\ref{lemma_1} leads to
\begin{align*}
    r_3 \lambda t_0 & \leq \int_{V} \left(\frac{\lambda}{|x|^{n-1}}\frac{1}{{a(x)}}\right) \left | \left \langle \frac{x}{|x|},\nabla w(x)  \right \rangle\right | {a(x)}\,dx \\ &
    \leq  \int_{V} \left( \frac{\lambda}{|x|^{n-1}}\frac{1}{a(x)}\right)^\frac{q}{q-1} a(x) \,dx + \int_{V}\left | \left \langle \frac{x}{|x|},\nabla w(x)   \right \rangle\right |^q  a(x)\,dx \\ & \leq r_1 \lambda^{\frac{q}{q-1}} + \int_{V} a(x)|\nabla w(x)|^q \,dx\,. 
\end{align*}
where $\lambda >0$ is fixed. Consequently,
\[ r_3 \lambda t_0\leq r_1 \lambda^{\frac{q}{q-1}} + \cF[w].\]
Then {for any $w\in u_0+C_0^\infty(B_1)$ it holds}
\begin{align*}
    \cF[w] & \geq r_1 \sup_{\lambda > 0} \left ( \lambda t_0 \frac{r_3}{r_1} - \lambda^{\frac{q}{q-1}}\right )  = r_1 \sup_{\lambda \in \R} \left ( \lambda t_0 \frac{r_3}{r_1} - |\lambda|^{\frac{q}{q-1}}\right )  = r_1 \left ( \frac{(q-1)t_0r_3}{qr_1}\right )^q \frac{1}{q-1}.
\end{align*}
Now, bearing in mind~\eqref{t} and using \eqref{estimate on the minimum} we get
\begin{equation}
     \inf_{w \in u_0 + C^{\infty}_0(B_1)} \cF[{w}] \geq \left ( \frac{r_3}{q}\right )^q \left ( \frac{q-1}{r_1}\right )^{q-1} t_0^q > r_2 t_0^p \geq  \inf_{v \in u_0 + W(B_1)} \cF[v]. \label{gap with c_infty}
\end{equation}
 Hence the occurrence of the Lavrentiev's phenomenon, that is \eqref{lavrentiev}, is proven.
 \end{proof}
\section{Generalizations}\label{sec:gen}
In this section, we describe how results presented in Section~\ref{sec:approximation-and-no-gap} may be generalized to consider a wider class of functionals.
\subsection{Variable exponent double phase functionals}
We can consider a variable exponent double phase functional, given by
\begin{equation}\label{eq:defVarDou}
    \mathcal{E}_1[u] = \int_{\Omega} b(x,u)\left(|\nabla u(x)|^{p(x)} + a(x)|\nabla u(x)|^{q(x)}\right)\,dx\,,
\end{equation}
where functions $p, q, a$ are such that $1 \leq p < q \in L^{\infty}(\Omega)$, $0 \leq a \in L^{\infty}(\Omega)$, and $b$ is continuous with respect to the second variable and $0 < \nu < b(\cdot, \cdot) < L$ for some constants $\nu, L$. The natural energy space for minimizers is\begin{align*}
    W_1(\Omega)&:=\left\{\vp\in W^{1,1}_0(\Omega):\quad \int_\Omega |\nabla \vp(x)|^{p(x)} + a(x)|\nabla \vp(x)|^{q(x)}\,dx<\infty\right\}\,.
    \end{align*} 
    Typical assumption imposed on the variable exponent is log-H\"older continuity. A function $p$ is said to be log-H\"older continuous (denoted $p\in\Plog(\Omega)$), if there exists $c>0$, such that for $x,y$ close enough it holds that
\[|p(x)-p(y)|\leq \frac{c}{\log\left({1}/{|x-y|}\right)}\,.\] 
    The results from~\cite{ibm, bgs-arma} state that $p, q \in \Plog$ and $a \in C^{0, \alpha}$ for $\alpha \geq \sup_{x} (q(x) - p(x))$ guarantees the absence of the Lavrentiev's phenomenon for functional~\eqref{eq:defVarDou}. This condition is meaningful only provided that $q(x) \leq p(x) + 1$ for every $x \in \Omega$. However, as for double-phase functional~\eqref{eq:defF}, we can assume only $a \in \SP^{\vk}$ for $\vk \geq \sup_{x} (q(x) - p(x))$ instead of $a \in C^{0, \alpha}$. That is, we have the following counterpart of Theorem~\ref{theo:lavrentiev-model}.
\begin{theo}\label{theo:lavrentiev-vardou}
    Let $\Omega \subset \rn$ be a bounded Lipschitz domain and let the functional $\cE_1$ be given by~\eqref{eq:defVarDou} for $p, q \in \Plog$, $1 < p < q$ and $a:\Omega\to\rp$. Suppose that $a \in \SP^{\vk}(\Omega)$ for $\vk > 0$. Let $u_0$ be such that $\cE_1[u_0] < \infty$. The following assertions hold true.
    \begin{enumerate}
        \item[(i)] If  $\vk \geq \sup_{x}(q(x)-p(x))$, then
            \begin{equation}\label{eq:janlav1}
        \inf_{v \in u_0 + W_1(\Omega)} \cE_1[v] = \inf_{w \in u_0 + C_c^\infty(\Omega)} \cE_1[w]\,. 
    \end{equation}
        \item[(ii)] Let $\gamma \in (0, 1]$. If $\vk \geq \left(\sup_{x}(q(x)-p(x))\right)(1-\gamma)$, then
            \begin{equation}\label{eq:janlav2}
        \inf_{v \in u_0 + W_1(\Omega) \cap C^{0, \gamma}(\Omega)} \cE_1[v] = \inf_{w \in u_0 + C_c^\infty(\Omega)} \cE_1[w]\,.
    \end{equation}
    \end{enumerate}
\end{theo}
One may also formulate corresponding counterparts of Theorem~\ref{theo:approx} and Theorem~\ref{theo:lavrentiev-general}. Note that the proof of Theorem~\ref{theo:approx} requires only modification of~\eqref{eq:jan7}.
\newline

%
%
%
\subsection{Orlicz multi phase functionals}

As in~\cite{basu, ibm, BaaBy} we can consider Orlicz multi phase functional, that is
\begin{equation}\label{eq:defOrlicz}
    \mathcal{E}_2[u] = \int_{\Omega} b(x,u)\left(\phi(|\nabla u(x)|) + \sum_{i=1}^{k}a_i(x)\psi_i(|\nabla u(x)|)\right)\,dx\,,
\end{equation}
where $\phi, \psi_i : [0, \infty) \to [0, \infty)$ are Young functions that satisfy $\Delta_2$ condition, $\lim_{t \to \infty} \frac{\psi_i(t)}{\phi(t)} = \infty$, $0 \leq a_i \in L^{\infty}(\Omega)$, for every $i=1,2,\dots, k$, and $b$ is continuous with respect to the second variable and $0 < \nu < b(\cdot, \cdot) < L$ for some constants $\nu, L$.  The natural energy space for minimizers is
\begin{align*}
    W_2(\Omega)&:=\left\{\vp\in W^{1,1}_0(\Omega):\quad \int_\Omega \phi(|\nabla \vp(x)|) + \sum_{i=1}^{k}a_i(x)\psi_i(|\nabla \vp(x)|)\,dx<\infty\right\}\,.
    \end{align*}
To get the absence of the Lavrentiev's phenomenon in this case,  we can modify our definition of the space $\SP^{\vk}$. For an arbitrary increasing and continuous function $\omega : [0, \infty) \to [0, +\infty)$ satisfying $\omega(0)=0$, one can define the space $\SP_{\omega}$ such that
\begin{equation}\label{eq:defom}
    a \in \SP^{\omega}(\Omega)\iff \exists_{C > 0} \quad \forall_{x, y} \quad a(x) \leq C(a(y) + \omega(|x-y|))\,.
\end{equation}
If $\omega$ defines appropriate modulus of continuity, i.e., $\omega$ is concave and $\omega(0) = 0$, then the fact that $a \in \SP_{\omega}$ is equivalent to the existence of a function $\wt a$, being comparable with $a$, and having modulus of continuity $\omega$, that is, for some $C > 0$ it holds that
\begin{equation*}
    \forall_{x, y} \quad |\wt a(x) - \wt a(y)| \leq C\omega(|x-y|)\,.
\end{equation*}
If $\omega \in \Delta_2$ is not necessarily concave, then from the fact that $\omega^{-1}(a)$ is comparable to some Lipschitz function, we can infer that $a \in \SP^{\omega}$. \newline 
Using the definition~\eqref{eq:defom}, we can obtain the counterpart of Theorem~\ref{theo:lavrentiev-model} for the functional of type~\eqref{eq:defOrlicz}.
\begin{theo}\label{theo:lavrentiev-varorl}
    Let $\Omega \subset \rn$ be a bounded Lipschitz domain and let the functional $\cE_2$ be given by~\eqref{eq:defOrlicz} for $\phi, \psi_i \in \Delta_2$, and $a_i:\Omega\to\rp$, for every $i=1,2,\dots,k$. Suppose that $a_i \in \SP^{\omega_i}(\Omega)$, where $\omega_i : \rp \to \rp$ is increasing and such that $\omega_i(0) = 0$ for every $i$. Let $u_0$ be such that $\cE_2[u_0] < \infty$. The following assertions hold true.
    \begin{enumerate}
        \item[(i)] If  $\omega_i(t) \leq \frac{\phi(t^{-1})}{\psi_i(t^{-1})}$ for every $i$, then
            \begin{equation}
        \inf_{v \in u_0 + W_2(\Omega)} \cE_2[v] = \inf_{w \in u_0 + C_c^\infty(\Omega)} \cE_2[w]\,. 
    \end{equation}
        \item[(ii)] Let $\gamma \in (0, 1]$. If $\omega_i(t) \leq \frac{\phi(t^{\gamma-1})}{\psi_i(t^{\gamma-1})}$ for every $i$, then
            \begin{equation}
        \inf_{v \in u_0 + W_2(\Omega) \cap C^{0, \gamma}(\Omega)} \cE_2[v] = \inf_{w \in u_0 + C_c^\infty(\Omega)} \cE_2[w]\,.
    \end{equation}
    \end{enumerate}
\end{theo}
    One may also formulate counterparts of Theorem~\ref{theo:approx} and Theorem~\ref{theo:lavrentiev-general}. Note that, similarly as in the previous section, the proof of Theorem~\ref{theo:approx} requires only modification of~\eqref{eq:jan7}. 
    Note that Theorem~\ref{theo:lavrentiev-varorl} improves the result of~\cite[Theorem 3.1]{BaaBy} by the use of the scale $\SP_{\omega}$ instead of $C^{\omega}$. In particular this allows to take into account Orlicz multi phase problems with $\limsup_{t\to\infty} t^{\sigma}\phi(t)/\psi_i(t)>0$ for arbitrary $\sigma>0$, which for $\sigma>1$ is excluded from the framework of~\cite{BaaBy}.

\subsection{Orthotropic case}
Our results may also be generalized to cover some types of orthotropic functionals. In particular, let us consider the orthotropic double phase functional given by
\begin{equation}\label{eq:defFort}
    \cE_3[u] = \sum_{i=1}^{n} \int_{\Omega} b_i(x,u)\left(|\partial_i u(x)|^{p_i} + a_i(x)|\partial_i u(x)|^{q_i}\right)\,dx\,,  
\end{equation}
where for every $i$ it holds that $1 < p_i < q_i$ and $0 \leq a_i \in L^{\infty}(\Omega)$, and $b_i$ is continuous with respect to the second variable and $0 < \nu < b(\cdot, \cdot) < L$ for some constants $\nu, L$, for every $i$.  The natural energy space for minimizers is
 \begin{align*}
    W_3(\Omega)&:=\left\{\vp\in W^{1,1}_0(\Omega):\quad \int_\Omega |\partial_i \vp(x)|^{p_i} + a_i(x)|\partial_i \vp(x)|^{q_i}\,dx<\infty\right\}\,.
    \end{align*}
We point out that for functionals satisfying such a decomposition, it is sufficient for the absence of the Lavrentiev's gap to look at each coordinate separately. In particular, we have the following theorem.
\begin{theo}
    Let $\Omega \subset \rn$ be a bounded Lipschitz domain and let the functional $\cE_3$ be given by~\eqref{eq:defFort}, where for every $i$ we have $1<p_i<q_i<\infty$ and $a_i:\Omega\to\rp$ that is allowed to vanish.  Suppose that for every $i$ it holds that $a_i \in \SP^{\vk_i}(\Omega)$ and $\vk_i > 0$. Moreover, let $u_0$ be such that $\cE_3[u_0] < \infty$.\\ The following assertions hold true.
    \begin{enumerate}
        \item[(i)] If  $\vk_i \geq q_i-p_i$ for every $i$, then
            \begin{equation}
        \inf_{v \in u_0 + W_3(\Omega)} \cE_3[v] = \inf_{w \in u_0 + C_c^\infty(\Omega)} \cE_3[w]\,. 
    \end{equation}
        \item[(ii)] Let $\gamma \in (0, 1]$. If $\vk_i \geq (q_i-p_i)(1-\gamma)$ for every $i$, then
            \begin{equation}
        \inf_{v \in u_0 + W_3(\Omega) \cap C^{0, \gamma}(\Omega)} \cE_3[v] = \inf_{w \in u_0 + C_c^\infty(\Omega)} \cE_3[w]\,.
    \end{equation}
    \end{enumerate}
\end{theo}
Corresponding counterparts of Theorem~\ref{theo:approx} and Theorem~\ref{theo:lavrentiev-general} may be formulated as well, together with counterparts for orthotropic versions of functionals~\eqref{eq:defVarDou} and~\eqref{eq:defOrlicz}.

\section*{Appendix} 


\begin{proof}[Proof of Proposition~\ref{prop:C1ainSPa+1}] We concentrate on {\it (i)}. Our reasoning is inspired by the proof of Glaeser-type inequality, see~\cite{Dol}. Suppose by contradiction that $a \not\in \SP^{1 + \alpha}$. This implies that there exist sequences $(x_k), (y_k) \subset \Omega$ and $C_k\in\R$ with $\lim_{k\to\infty}C_k=\infty$ such that
\begin{equation}\label{eq:mar1}
    a(x_k) \geq C_k(a(y_k) + |x_k-y_k|^{1+\alpha})\,.
\end{equation}
As $\overline{\Omega}$ is compact, by taking subsequences if necessary,  we may assume that $x_k \to \bar x, y_k \to \bar y$, where $\bar x, \bar y \in \overline{\Omega}$. Observe that taking limits in~\eqref{eq:mar1}, we obtain that for every $C>0$ we have $a(\bar x) > C \cdot (a(\bar y) + |\bar x-\bar y|^{1+\alpha})$. As $a$ is bounded, we have that $a(\bar y) + |\bar x-\bar y|^{1+\alpha} = 0$. That is, we have $\bar x = \bar y$ and $a(\bar x) = 0$. We shall denote $x_0: = \bar x =\bar  y$. As $a(x_0) = 0$, by assumption, we have $x_0 \in \Omega$ and  there exist $R > 0$ such that $B(x_0, R) \subseteq \Omega$. 

Let us fix any $\nu \in \rn$ such that $|\nu| = 1$. By Lagrange Mean Value Theorem, for arbitrary $z \in B(x_0, R)$ and $h \in \R$ such that $z + h\nu \in B(x_0, R)$, we have 
\begin{equation}\label{eq:mar2}
    a(z + h\nu) = a(z) + h\frac{\partial a}{\partial \nu}(z + \varsigma \nu),
\end{equation}
where $\varsigma \in [-|h|,|h|]$ .  Using that $a \in C^{1, \alpha}(\Omega)$, we get that for some constant $C$, independent of $\nu$, we have
\begin{equation*}
    \left|\frac{\partial a}{\partial \nu}(z + \varsigma \nu) - \frac{\partial a}{\partial \nu}(z)\right| \leq C|\varsigma|^{\alpha} \leq C|h|^{\alpha}\,,
\end{equation*}
and, consequently,
\begin{equation*}
    \frac{\partial a}{\partial \nu}(z) - C|h|^{\alpha} \leq \frac{\partial a}{\partial \nu}(z + \varsigma \nu) \leq \frac{\partial a}{\partial \nu}(z) + C|h|^{\alpha}\,.
\end{equation*}
Thus, for $h \geq 0$ it holds that
\begin{equation*}
    h\frac{\partial a}{\partial \nu}(z + \varsigma \nu) \leq h\frac{\partial a}{\partial \nu}(z) + Ch|h|^{\alpha} = h\frac{\partial a}{\partial \nu}(z) + C|h|^{\alpha+1}\,,
\end{equation*}
while for $h < 0$ we have
\begin{equation*}
    h\frac{\partial a}{\partial \nu}(z + \varsigma \nu) \leq h\frac{\partial a}{\partial \nu}(z) - Ch|h|^{\alpha} = h\frac{\partial a}{\partial \nu}(z) + C|h|^{\alpha+1}\,.
\end{equation*}
%
By~\eqref{eq:mar2} and the last two displays, it means that
\begin{equation*}
    a(z + h\nu) \leq a(z) + h\frac{\partial a}{\partial \nu}(z) + C|h|^{1+\alpha}\,. 
\end{equation*}
As $a \geq 0$, we have
\begin{equation}\label{eq:mar3}
    0 \leq a(z) + h\frac{\partial a}{\partial \nu}(z) + C|h|^{1+\alpha}\,,
\end{equation}
as long as $h \in \R$ and $z, z + h\nu \in B(x_0, R)$. 

For any $z \in B(x_0, R)$, let us now denote
\begin{equation*}
h_z \coloneqq -c\left|\frac{\partial a}{\partial \nu}(z)\right|^{1/\alpha} \text{sgn}\left(\frac{\partial a}{\partial \nu}(z)\right)\ \text{ with }\ c = (2C)^{-1/\alpha}\,.\end{equation*}
Note that as $a(x_0) = 0$, we also have $\nabla a(x_0) = 0$, as $x_0 \in \Omega$ is a minimum of $a$. Since $a \in C^{1, \alpha}(\overline{\Omega})$, for any $z \in B(x_0,R)$ we have $\left|\frac{\partial a}{\partial \nu}(z)\right| \leq C|z-x_0|^{\alpha}$, which gives us
\begin{equation*}
    |z+h_z\nu-x_0| \leq |z-x_0| + |h_z| = |z-x_0| + c\left|\frac{\partial a}{\partial \nu}(z)\right|^{1/\alpha} \leq (1+cC^{1/\alpha})|z-x_0|\,.
\end{equation*}
Therefore, if we take $r \coloneqq \frac{R}{1 + cC^{1/\alpha}}$, for any $z \in B(x_0, r)$ we have $z+h_z\nu\in B(x_0,R)$. Hence, by \eqref{eq:mar3} we obtain%
\begin{equation*}
    0 \leq a(z) + h_z\frac{\partial a}{\partial \nu}(z) + C|h_z|^{1+\alpha} = a(z) - \tfrac{1}{2}c\left|\frac{\partial a}{\partial \nu}(z)\right|^{1 + 1/\alpha}\,,
\end{equation*}
which means that for some constant $C_a > 0$ it holds that \begin{equation}\label{eq:maininq}
    \left|\frac{\partial a}{\partial \nu}(z)\right| \leq C_aa(z)^{\frac{\alpha}{1 + \alpha}}\,.
\end{equation}
Note that by ambiguity of $\nu$, estimate~\eqref{eq:maininq} holds for arbitrary $\nu \in \rn$ such that $|\nu| = 1$.

Let us take any $x, y \in B(x_0, r)$. Note that we can always find $\tilde{y} \in [y, x]$ such that $a(\tilde{y}) \leq a(y)$ and $a > 0$ on the segment $(y, x)$. Indeed, if $a > 0$ on $(y,x)$, then we can take $\tilde{y} = y$. In other case, we may define
\begin{equation*}
    \tilde{t} \coloneqq \sup \{t \in [0, 1] : a(y + t(x-y)) = 0\}
\end{equation*}
and set $\tilde{y} \coloneqq y + \tilde{t}(x-y)$. We see by the definition that $a(\tilde{y}) = 0 \leq a(y)$ and $a$ is positive on $(y, x)$. Therefore, if we set $\nu = \frac{x - \tilde{y}}{|x - \tilde{y}|}$, the function $t \mapsto a(\tilde{y} + t\nu)^{\frac{1}{1+\alpha}}$ is differentiable for $t \in (0, |x-\tilde{y}|)$, with derivative equal to $\left(\frac{\partial a}{\partial \nu}(\tilde{y}+t\nu)\right)  (a(\tilde{y} + t\nu))^{-\frac{\alpha}{1+\alpha}}$. By the definition of $\tilde{y}$ and~\eqref{eq:maininq}, we have
\begin{align*}
    a(x)^{\frac{1}{1+\alpha}}-a(y)^{\frac{1}{1+\alpha}} \leq a(x)^{\frac{1}{1+\alpha}}-a(\tilde{y})^{\frac{1}{1+\alpha}} &= \int_{0}^{|x-\tilde{y}|} \left(\frac{\partial a}{\partial \nu}(\tilde{y}+t\nu)\right)  (a(\tilde{y} + t\nu))^{-\frac{\alpha}{1+\alpha}} \,dt \leq C_a|x-\tilde{y}| \leq C_a|x-y|\,,
\end{align*}
which by symmetry means that $a^{\frac{1}{1+\alpha}}$ is Lipschitz on $B(x_0, r)$. By Remark~\ref{rem:zk}, we have that $a \in \SP^{1+\alpha}(B(x_0, r))$, which contradicts~\eqref{eq:mar1}, as $(x_k)_k$ and $(y_k)_k$ converge to $x_0$. Hence, $a \in \SP^{1+\alpha}(\Omega)$.\newline

For {\it (ii)} it is enough to consider $x_0\in\Omega\subset\rn$ and $a(x)=|x-x_0|^2 $, which is smooth, but only in $\SP^2$.
\end{proof}

\section*{Acknowledgement}
 We would like to express gratitude to Pierre Bousquet for fruitful discussions that in particular drawn our attention to issues solved in Proposition~\ref{prop:C1ainSPa+1}.

The project started with discussions of all authors during Thematic Research Programme Anisotropic and Inhomogeneous Phenomena at University of Warsaw in 2022.

\printbibliography
\end{document}